\documentclass[11pt,leqno]{article}
\usepackage{amsmath,amssymb,amsthm,nicefrac,mathrsfs}
\usepackage[colorlinks]{hyperref}
\usepackage[linesnumbered]{algorithm2e}
\usepackage{tabu}
\usepackage{bbm}
\usepackage{comment}
\usepackage{calrsfs}
\usepackage[mathscr]{euscript}
\usepackage{hyperref}
\usepackage{indentfirst}
\usepackage{mathrsfs}
\usepackage{physics}
\usepackage{lipsum}
\usepackage{verbatim}
\usepackage[a4paper, margin=4cm, footskip=1cm]{geometry}
\usepackage{caption}
\usepackage[T1]{fontenc}
\usepackage{booktabs}
\usepackage{siunitx}
\usepackage{xcolor}
\usepackage{graphicx}
\usepackage{subcaption}
\usepackage{scalerel,stackengine}
\stackMath
\newcommand\reallywidehat[1]{%
	\savestack{\tmpbox}{\stretchto{%
			\scaleto{%
				\scalerel*[\widthof{\ensuremath{#1}}]{\kern-.6pt\bigwedge\kern-.6pt}%
				{\rule[-\textheight/2]{1ex}{\textheight}}
			}{\textheight}%
		}{0.5ex}}%
	\stackon[1pt]{#1}{\tmpbox}%
}

\newcommand{\Addresses}{{
		\bigskip
		\footnotesize

	\noindent 	NGARTELBAYE GUERNGAR \\
		\textsc{Department of Mathematics, University of North Alabama,
		Florence, AL 35632}\\
		\textit{E-mail address}: \texttt{nguerngar@una.edu}\\
		\textit{URL}: \texttt{\url{	https://www.researchgate.net/profile/Ngartelbaye_Guerngar}}
	
		\medskip

	\noindent ERKAN NANE \\
		 \textsc{Department of Mathematics and Statistics, Auburn University,
		 	Auburn, AL 36849}\\
		\textit{E-mail address}: \texttt{ezn0001@auburn.edu}\\
		\textit{URL}: \texttt{\url{http://www.auburn.edu/~ezn0001}}
		
		\medskip

		\noindent RAMAZAN TINATZTEPE\\
		\textsc{Deanship of preparatory year and supporting studies, Imam Abdulrahman Bin Faisal University, Damman, KSA }\\
		\textit{E-mail address}: \texttt{ttinaztepe@iau.edu.sa}\\

		\noindent S\"ULEYMAN ULUSOY\\
		\textsc{ Department of Mathematics and Natural Sciences, American University of Ras Al Khaimah, Ras Al Khaimah, UAE}\\
		\textit{E-mail address}: \texttt{suleyman.ulusoy@aurak.ac.ae}\\
		\textit{URL}: \texttt{\url{	https://www.aurak.ac.ae/en/dr-suleyman-ulusoy}}

			\medskip
		\noindent HANS-WERNER VAN WYK \\
		\textsc{Department of Mathematics and Statistics, Auburn University,
			Auburn, AL 36849}\\
		\textit{E-mail address}: \texttt{hzv0008@auburn.edu}\\
		\textit{URL}: \texttt{\url{http://www.auburn.edu/~hzv008}}
		
		\medskip

}}

\newtheorem{theorem}{Theorem}[section]
\newtheorem{definition}[theorem]{Definition}
\newtheorem{lemma}[theorem]{Lemma}

\theoremstyle{definition}
\newtheorem{example}[theorem]{Example}
\numberwithin{equation}{section}

{%

\begin{enumerate}}%
{\end{enumerate}
}

\def\be{\begin{equation}}
\def\ee{\end{equation}}
\def\ba{\begin{aligned}}
\def\ea{\end{aligned}}
\def\bes{\begin{equation*}}
\def\ees{\end{equation*}}

\def\bc{\begin{cases}}
\def\ec{\end{cases}}
\numberwithin{equation}{section}
\everymath{\displaystyle}

\author{Ngartelbaye Guerngar\\
	University of North Alabama\\
	\and Erkan Nane\\ Auburn University\\ \and Ramazan Tinatztepe\\ Imam Abdurahman Bin Faisal University  \\ \and Suleyman Ulusoy\footnote{The research of S.U. has been partially supported by BAGEP 2015 award.}\\American University of Ras Al Khaimah\\ \and Hans Werner Van Wyk\\ Auburn University	
}
\title{Simultaneous inversion for the fractional exponents in the
	space-time fractional diffusion equation $\partial_t^\beta u= -\big(-\Delta\big)^{\alpha/2}u-\big(-\Delta\big)^{\gamma/2}u$
	\date{}
}

\graphicspath{{./fig/}}

\begin{document}
\maketitle






\begin{abstract}

\noindent In this article, we consider the space-time fractional (nonlocal)  equation characterizing the so-called "double-scale" anomalous diffusion
$$\partial_t^\beta u(t, x) = -(-\Delta)^{\alpha/2}u(t,x) - (-\Delta)^{\gamma/2}u(t,x) \ \ t>0, \ -1<x<1, $$
where $\partial_t^\beta$ is the Caputo fractional derivative of order $\beta \in (0,1)$ and  $0<\alpha\leq \gamma<2.$ We consider
a nonlocal inverse problem and show that the fractional exponents $\beta$, $\alpha$ and $\gamma$ are determined uniquely by the data $u(t, 0) = g(t), \ 0 < t \leq  T.$ The existence of the solution for the inverse problem is proved  using the quasi-solution method which is based
on minimizing an error functional between the output data and the additional
data. In this context, an input-output mapping is defined and its continuity 
 is established. The uniqueness of the solution for the inverse problem is  proved by means of  eigenfunction expansion of the solution of the forward problem and some basic properties of fractional Laplacian. A numerical method based on discretization of the minimization problem, namely the steepest descent method and a least squares approach,  is proposed for the solution of the inverse problem. The numerical method determines the fractional exponents simultaneously. Finally, numerical examples with noise-free and noisy data illustrate applicability and
high accuracy of the proposed method.
\end{abstract}

\newpage

\section{Introduction}
In this article, we study an inverse problem associated with the following space-time fractional diffusion equation 
\begin{equation}\label{eq1}
\begin{cases}
  &\partial_t^\beta u(t, x) = -(-\Delta)^{\alpha/2}u(t,x) - (-\Delta)^{\gamma/2}u(t,x), ~0<t<T,~ -1<x<1,\\ 
  & u(t, 1) =  u(t,-1)=0, ~~0 < t < T, \\ 
  & u(0, x) = f(x), ~-1 < x < 1. \\ 
\end{cases}
\end{equation}
Here $T > 0$ is the final time and $f$ is a "nice" function in $L^2(D)$, called the initial function, where $D\subset\mathbb{R}$ is the open unit ball. Here the fractional exponents of the Laplacian satisfy\\ $0<\alpha\leq\gamma<2.$ $\beta \in (0, 1) $ and $\partial_t^\beta$ is the Caputo fractional time-derivative. It is defined as

\begin{equation}\label{CapDer}
\partial_t^\beta q(t,\cdot)= \frac{\partial^\beta q(t,\cdot)}{\partial t^\beta}:= \frac{1}{\Gamma(1-\beta)}\int_0^t \frac{\partial q(s,\cdot)}{\partial s}\frac{ds}{(t-s)^\beta},
\end{equation}
where $\Gamma(.)$ is the Euler's gamma function. For example, $\partial_t^\beta (t^p)= \frac{t^{\beta-p}\Gamma(p+1)}{\Gamma(p+1-\beta)}$ for any $p>0$. This definition of the Caputo fractional derivative is intended to properly handle initial values \cite{Cap,CMN,10}, since its Laplace transform $s^\beta \tilde{q}(s,\cdot)-s^{\beta -1} q(0, \cdot)$ incorporates the initial value in the same way the first derivative does. Here, $\tilde{q}(s,\cdot)=\int_0^\infty e^{-ts} q(t,\cdot)dt$ represents the usual Laplace transform of the function $q$. 
It is also well known that, if $q\in C^1(0,\infty)$ satisfies $\big|q'(t)\big|\leq C t^{\nu -1}$ for some $\nu>0$, then by \eqref{CapDer}, the Caputo derivative of $q$ exists for all $t>0$ and the derivative is continuous in $t>0$ \cite{Kil, 11}. 

\noindent The following class of functions will play an important role in this article.
\begin{definition}\label{MtgLfl}
	The Generalized (two-parameter) Mittag-Leffler function is defined by:
	\begin{equation}\label{mtglfr}
	E_{\beta,\alpha}(z)=\sum_{k=0}^{\infty}\frac{z^k}{\Gamma(\beta k+\alpha)}, \ \ z\in\mathbb{C}, \ \ \ \Re (\alpha)>0, \ \ \ \Re (\beta)>0,
	\end{equation}
	where  $\Re(\cdot)$ is the real part of a complex number.  When $\alpha=1$, this function reduces to $E_{\beta}(\cdot):=E_{\beta,1}(\cdot).$
\end{definition}
\noindent It is well-known that the Caputo derivative has a continuous spectrum \cite{CMN,11}, with eigenfunctions given in terms of the Mittag-Leffler function.  In fact, it is not hard to check that the function $q(t)=E_\beta(-\lambda t^\beta)$ is a solution of the eigenvalue equation
$$
\partial_t^\beta q(t)=-\lambda q(t) \ \ \text{for any} \ \lambda>0.
$$

\noindent For $0<\nu< 2$, $(-\Delta)^{\nu/2}$ denotes the fractional Laplacian. Here, we define it using the spectral decomposition of the Laplacian. Let $\big(\bar{\mu}_k, \psi_k\big)$ be the eigenpair corresponding to the Helmholtz's equation

\begin{equation}\label{Helmh}
\begin{cases}
-\Delta \psi_k= \bar{\mu}_k \psi_k \ \ \text{in} \ D\\
\psi_k=0 \ \ \ \text{on }\ \ \partial D.
\end{cases}
\end{equation} 
A simple calculation shows that $\bar{\mu}_k= \Bigg(\frac{k\pi}{2}\Bigg)^2 $ and 
$
\psi_k= \sin\Big[\frac{k\pi}{2}(x+1)\Big] \ \ \text{for all}\  k\geq 1.
$
  For $0<\alpha\leq \gamma<2$,  define the operator  $\mathtt{L}^{\alpha,\gamma}_D:=-(-\Delta)^{\alpha/2}-(-\Delta)^{\gamma/2}$ on $D$ for 

$$
f\in \text{Dom} \Big(\mathtt{L}^{\alpha,\gamma}_D\Big)= \Big\{ f= \sum_{n=1}^\infty c_n\psi_n\in L^2(D): \sum_{n=1}^\infty c_n^2 \mu_n^2<\infty \Big\}:= \dot{H}^{\alpha,\gamma}
$$
and
\begin{equation}\label{FrLpDef}
 \mathtt{L}^{\alpha,\gamma}_Df(x)=-\sum_{n=1}^\infty c_n\mu_n\psi_n(x) \ \ \text{with }\ \mu_m=\bar{\mu}_m^{\alpha/2}+ \bar{\mu}_m^{\gamma/2} \ \ \text{for all} \ m=1,2,\cdots.
\end{equation}
Note that for all $k\geq 1$, the eigenpair $(\mu_k, \psi_k)$ is such that    $0<\mu_1\leq \mu_2\leq \cdots$ is a sequence of positive numbers and $\Big(\psi_k\Big)_{k\geq 1}$ is an orthonormal basis of $L^2(D).$
Clearly $\dot{H}^{\alpha,\gamma}\subset L^2(D)$. It is a Hilbert space endowed with the  inner product, $\langle\cdot,\cdot\rangle$ represents the usual inner product on $L^2(D)$, \\${\langle u,v\rangle}_{\dot{H}^{\alpha,\gamma}}= \langle \mathtt{L}^{\alpha,\gamma}_D u, \mathtt{L}^{\alpha,\gamma}_D v\rangle$
 and induced norms ${\|v\|}_{\dot{H}^{\alpha,\gamma}}= {\big\|\mathtt{L}^{\alpha,\gamma}_Dv\big\|}_{L^2(D)}= \Bigg[\sum_{n=1}^\infty \mu_n^2{\langle v,\psi_n\rangle}^2\Bigg]^{1/2}$.\\
 For example, $\dot{H}^{0,0}= L^2(D)$,  $\dot{H}^{1,1}= H^1_0(D)$ and $\dot{H}^{2,2}= H^2(D)\cap H^1_0(D)$ with equivalent norms and $\dot{H}^{-\alpha,-\gamma}$ can be identified with the dual space $\Big(\dot{H}^{\alpha,\gamma}\Big)^*$ for $\alpha, \gamma>0. $ Let ${\langle f, \phi \rangle}_{*}$ denote the value of $f$ operating on the bounded linear functional $\phi\in \dot{H}^{\alpha,\gamma}$. It turns out that $\dot{H}^{-\alpha,-\gamma}$ is also a Hilbert space with the norm ${\|\phi\|}_{\dot{H}^{-\alpha,-\gamma}}= \Bigg[\sum_{n=1}^\infty \mu_n^{-2}|{\langle f, \psi_n \rangle}_{*}|^2\Bigg]^{1/2}$. Moreover, ${\langle f, \phi \rangle}_{*}=\langle f, \phi\rangle$ if $f\in L^2(D)$ and $\phi\in \dot{H}^{\alpha,\gamma} $, see for example \cite[Chap. V]{Brez}.

The main purpose of this article is to determine simultaneously the fractional  exponents $\beta, \alpha$ and $\gamma$ in \eqref{eq1} by means of the observation data $u(t, 0) =
g(t), \ 0 < t \leq  T .$  By this result, one can expect that by means of experiments the important
parameters $\beta, \alpha$ and $\gamma$ characterizing the "double-scale" anomalous diffusion \eqref{eq1} can be identified simultaneously.

\noindent In fact, it is well-known that the traditional diffusion equation $\partial_t u= \Delta u$ describes a cloud of spreading particles at the macroscopic level and  the  space-time fractional diffusion equation $\partial_t^\beta u= -(-\Delta)^{\alpha/2}u$ with $0<\beta<1$ and $0<\alpha<2$ models anomalous diffusions \cite{CMN,MBSB,GuNaSoy}. Here, the classical Laplacian $(-\Delta)$ is the generator of a Brownian motion and  the fractional Laplacian $(-\Delta)^{\alpha/2}$ is the infinitesimal  generator of a symmetric $\alpha-$ stable process \\ $X= \Big\{  X_t, \ t\geq 0, \mathbb{P}_x, \ x\in \mathbb{R}^d \Big\}$, a typical example of a non-local operator. This process is  a L\'evy process satisfying

$$
\mathbb{E}\Big[ e^{i\xi(X_t-X_0)}\Big]= e^{-t|\xi|^\alpha} \ \ \ \text{for every} \ x, \xi\in \mathbb{R}^d.
$$
 Now, suppose $X$ is a Brownian motion and let $X^D$ denote the "killed" process, i.e
 
 \begin{equation}
 X^D_t:= \begin{cases}
 &X_t , \ \ \ \ \  t<\tau_D \\
 &\partial,  \ \ \ \ t\geq \tau_D.
 \end{cases}
 \end{equation}
 Here, 
 \begin{equation}\label{tau}
 \tau_D:=\inf\{ t\geq 0: X_t\in \partial D\}
 \end{equation}
 is the first existing time and  $\partial$ is a cemetery  point added to $D$. Throughout this paper, we use the convention that any real-valued function $f$ can be extended by taking $f(\partial)=0.$
  Then   $\Big[(-\Delta)^{\alpha/2}+ (-\Delta)^{\gamma/2}\Big]\Bigg|_D$ is the infinitesimal generator of the process $X^D(E_t)$, where $E_t$ is the inverse stable subordinator with Laplace exponent $\phi(s)= s^{\alpha/2}+s^{\gamma/2}$.
The L\'evy process $X^D(E_t)$ runs on two different scales: on the small spatial scale, the $\gamma$ component dominates, while on the large spatial scale the $\alpha$ component takes over \cite{ChKmSg}.
 
 There have been  recently  many works in inverse problems with fractional derivatives. However,  most of the problems considered involve only a fractional time derivative and the  determination of that fractional exponent under some additional condition(s) is the inverse problem. In fact, these problems are physically and practically very important \cite{CNYY, JR, Mis2, Mis1, Mis5, Mis6,  Mis3, Mis4, ZX}. The current study extends the work of \cite{TrPeSy,TarSoy} in which fractional exponents were considered both in the time and space variable. It is good to note that this is a very recent approach in the inverse problems community, see \cite{LiZhg} and some of the references cited therein. This study can be regarded as a continuation of  \cite{GuNaSoy} using a spectral eigenfunction expansion of the weak solution to the initial/boundary value problem \eqref{eq1}.

The rest of this  article is organized as follows: in the next section we provide a quick analysis  of the direct problem and introduce the inverse problem. Section \ref{Sect3} is devoted to both the statement and the proof of our main results. Section \ref{Sect4} provides some details on the algorithm used to obtain the solution to our problem and Section \ref{Sect5} concludes this article with some numerical examples to  illustrate the applicability and the high accuracy of the  method used. Throughout this article, the letter $c$, in upper or lower case, with or without a subscript, denotes a constant whose value is not of interest in this article and may stay the same or change from line to line.

\section{The direct and inverse problem}
In this section, we provide a quick analysis of the direct problem and introduce the inverse problem. We begin with a definition:
\begin{definition}\label{WeakSol}
We call $u$ a weak solution to \eqref{eq1} if the following conditions are satisfied:
\begin{equation}
\begin{split}
&u(t,.) \in \dot{H}^{\alpha,\gamma}\ \ \text{for each} \ t>0,\\
& u(t, \cdot)= 0 \ \text{on } \ \partial D \ \text{for each } \ 0<t\leq T\\
&\lim\limits_{t\downarrow 0} u(t,x)= f(x) \ \ a.e, \\
& \partial_t^\beta u(t,x)= \mathtt{L}_D^{\alpha,\gamma} u(t,x) \ \ \text{in} \ L^2(D)
\end{split}
\end{equation}
\end{definition}
 It is well-know that the direct problem \eqref{eq1} has a unique weak solution given by the eigenfunction expansion \cite{Mis6, CMN, Luc, GuNaSoy}
\begin{equation}\label{WkSol}
u(t,x)=\sum_{n=1}^{\infty} E_\beta(-\mu_nt^\beta)\langle f, \psi_n\rangle \psi_n(x).
\end{equation}
 
We will show later in this article that the series in \eqref{WkSol} exists, is unique and is uniformly convergent in $C\Big( (0,T];\dot{H}^{\alpha,\gamma}\Big) .$ 

\noindent Set $u(t):= u(t,\cdot)$. In the existence and uniqueness theorem, we will need the solution of our problem in the following form: 

\begin{equation}\label{Eq1ODE}
\begin{cases}
&\partial_t^\beta u(t)= \Big[-(-\Delta)^{\alpha/2}-(-\Delta)^{\gamma/2}\Big]u(t)+h(t), \ \ \ t>0\\
& u(0)=g.
\end{cases}
\end{equation}
To this aim, we define

$$
U(t)g= \sum_{n=1}^{\infty} E_\beta(-\mu_nt^\beta)\langle g, \psi_n\rangle \psi_n(x), \ \ \ t\geq 0
$$

and 

\begin{equation}\label{Vg}
V(t)g= t^{\beta -1}\sum_{n=1}^{\infty} E_{\beta, \beta}(-\mu_nt^\beta)\langle g, \psi_n\rangle \psi_n(x), \ \ \ t\geq 0.
\end{equation}

Then a solution of \eqref{Eq1ODE} is given by

\begin{equation}\label{SolEq1ODE}
u(t)= U(t)g+\int_0^t V(t-s)h(s)ds, \ \ \ t>0.
\end{equation}

\noindent The following lemmas indicate important properties of  Mittag-Leffler functions. They  will be used frequently in the sequel.
\begin{lemma}\label{Mtg1}
If $0<\beta<2$, $\mu$ is such that  $\pi\beta/2<\mu<\min(\pi, \pi\beta)$ and $\mu\leq |\arg(z)|\leq \pi$, then the following expansion holds
	\begin{equation}\label{Mitg1}
	\big|E_{\beta}(-z)\big|=\frac{1}{z\Gamma(1-\beta)}+ O(|z|^{-2}) .
	\end{equation}
\end{lemma} 
 
\begin{lemma}\label{Mtg2}
	For each $0<\beta<2$, $\nu$  a complex number such that $\Re(\nu)>0$, \\  $\pi\beta/2<\mu<\min(\pi, \pi\beta)$ and $\mu\leq |\arg(z)|\leq \pi$, there exists a constant $C_0>0$ such that
	\begin{equation}\label{Mitg2}
	\big|E_{\beta, \nu}(z)\big|\leq \frac{C_0}{1+|z|}, \ \ .
	\end{equation}
\end{lemma}
\begin{lemma}\label{Mtg3}
If $0\leq \beta\leq 1$, then $E_\beta(-z)$ is completely monotone on $(0,\infty)$ and all the derivatives of $E_\beta(-z)$ are bounded on $(0,\infty)$.
\end{lemma}

\noindent The following theorem gives the regularity of the solution of the direct problem.
\begin{theorem}
Let $f\in \dot{H}^{\alpha,\gamma} \subset L^2(D)$. Then there exists a unique weak solution $u$ of \eqref{eq1} such that $u\in C\Big([0,T]; L^2(D\Big)\cap C\Big((0,T]; \dot{H}^{\alpha,\gamma}\Big)$. Moreover, there exists a positive constant $C$ such that $\partial_t^\beta u\in C\Big((0,T]; L^2(D)\Big)$ and 
\begin{equation}\label{RegSol}
{\|u(t,\cdot)\|}_{\dot{H}^{\alpha,\gamma}}+ {\|\partial_t^\beta u(t,\cdot)\|}_{L^2(D)}\leq C t^{-\beta} {\|f\|}_{\dot{H}^{\alpha,\gamma}}.
\end{equation}
\end{theorem}
\begin{proof}
The series in \eqref{WkSol} is certainly a weak solution to \eqref{eq1}. The existence of this series  is proved in \cite[(2.3)]{GuNaSoy}.\\ For the uniqueness  of a weak solution to \eqref{eq1}, it is enough to show that a function $u$ in Definition \ref{WeakSol} solving \eqref{eq1} with $f=0$ must be $u\equiv 0$ . We follow a similar argument  from \cite{Mis6} to this aim. Since $\psi_n$ are the eigenfunctions of the following eigenvalue problem:
\begin{equation}\label{EigenVPrb}
	\begin{cases}
	&\mathtt{L}^{\alpha,\gamma}_D \psi_n=-\mu_n\psi_n \ \ \text{in} \ \ D\\
	 &\psi_n=0 \ \qquad\qquad \text{on} \ \partial D,
	\end{cases}
\end{equation}
in terms of the regularity of $u$, taking the duality pairing ${\langle .,. \rangle}_{*}$ of the first equation in \eqref{eq1} with $\psi_n$ and setting $u_n(t):={\langle u(t, \cdot), \psi_n \rangle}_{*}$, we obtain 
\begin{equation}\label{PairEq}
\partial_t^\beta u_n(t)=-\mu_n u_n(t)  \ \ \text{for almost every} \ t\in (0, T].
\end{equation}
Since $u(t, \cdot)\in L^2(D)$ for almost every $t\in (0, T]$ and $u_n(t)={\langle u(t, \cdot), \psi_n \rangle}_{*}= {\langle u(t, \cdot), \psi_n \rangle} $. It follows from $\lim\limits_{t\downarrow 0}{\|u(t,\cdot)\|}_{\dot{H}^{-\alpha,-\gamma}}=0$ that $u_n(0)=0$. Thus, due to the existence and uniqueness of the solution to the ordinary fractional differential equation \eqref{PairEq}, see for example \cite[Chap 3]{ 11} , it must be the case that $u_n(t)\equiv0$ for $n=1,2, \cdots$. Finally, since $\{ \psi_n\}_{n\in\mathbb{N}}$ is a complete orthonormal system of $L^2(D),$ we have $u\equiv0$ in $(0,T]\times D.$

 We now provide a proof for the estimate \eqref{RegSol}. Note that 
\begin{equation}
{\big\|u(t,\cdot)\big\|}_{L^2(D)}^2= \sum_{n=1}^{\infty} \Big|E_\beta(-\mu_nt^\beta)\langle f, \psi_n\rangle \Big|^2\leq \sum_{n=1}^{\infty} c_1\langle f, \psi_n\rangle^2\leq c_2{\big\|f\big\|}_{L^2(D)}^2\leq C_2{\big\|f\big\|}_{\dot{H}^{\alpha,\gamma}}^2.
\end{equation}
Moreover, by Lemma \ref{Mtg2},
\begin{equation}\label{LplInq}
{\big\|\partial_t^\beta u(t,\cdot)\big\|}^2_{L^2(D)}= \sum_{n=1}^\infty \Big|\mu_n E_\beta(-\mu_nt^\beta)\langle f, \psi_n\rangle \Big|^2\leq c_3 t^{-2\beta}{\big\|f\big\|}_{L^2(D)}^2\leq C_3 t^{-2\beta}{\big\|f\big\|}_{\dot{H}^{\alpha,\gamma}}^2. 
\end{equation}
In particular, \eqref{LplInq} implies that 
\begin{equation}\label{IndLplInq}
{\big\| (-\Delta)^{\xi/2} u\big\|_{L^2(D)}}\leq  c_4 t^{-\beta}{\big\|f\big\|}_{L^2(D)}, \ \ \   0<\xi\leq \gamma.
\end{equation}
Next, since the series in \eqref{WkSol} converges uniformly in $t\in [0,T],$ we see that $u\in  C\Big((0,T]; L^2(D)\Big).$ Moreover, in \eqref{LplInq}, since the series $-\sum_{n=1}^\infty \mu_n E_\beta(-\mu_nt^\beta)\langle f, \psi_n\rangle \psi_n$ is uniformly convergent for $t\in [\delta,T]$ with any given $\delta>0,$ this implies  that $\mathtt{L}^{\alpha,\gamma}_D u\in C\Big((0,T]; L^2(D)\Big),$ i.e \\ $u\in C\Big((0,T]; \dot{H}^{\alpha,\gamma}\Big).$ Whence
$u\in  C\Big([0,T]; L^2(D)\Big)\cap C\Big((0,T]; \dot{H}^{\alpha,\gamma}\Big) $ and \eqref{LplInq} holds.
\end{proof}
Next we define the inverse problem. As it is known, a direct problem aims to find a
solution that satisfies a  given differential equation (ordinary, partial or fractional) and related initial and boundary conditions. In some problems, the main equation and the conditions are not sufficient to obtain a solution, but instead some additional conditions (also called measured output data) are required. Such problems are called the corresponding inverse  problems. In general, the additional conditions may be given on the domain's boundary, on the final time or on the whole domain (also known as nonlocal condition). In this paper, we use the following additional condition
\begin{equation}\label{AddData}
u(t,0)= g(t), \ \ \ 0<t\leq T.
\end{equation}
The inverse problem here consists of determining the unknown fractional orders  $\beta, \alpha $ and $\gamma$ of problem \eqref{eq1} from the  additional condition \eqref{AddData}. For some technical reasons in the proof of existence and uniqueness of a solution to our inverse problem, we require  the initial condition in \eqref{eq1} to satisfy either
\begin{equation}\label{RestF}
\langle f, \psi_n\rangle>0  \ \ \text{for} \ \ n\geq 1\ \text{and}\ n \ \text{odd} \ \ \ \text{ or} \ \  \langle f, \psi_n\rangle <0 \ \text{for} \ n\geq 1, \ \text{and}\ n \ \text{ odd}.
\end{equation}

To the best of our knowledge, there are not many works related to inverse
problems for the fractional diffusion equations involving fractional Laplacian, see \cite{GuNaSoy, TarSoy, TrPeSy}. Our current paper makes some contribution to this subject.

The next section is devoted to our main results, i.e the statement and the proof of the existence and uniqueness theorem for our inverse problem.
\section{Statement and proof of main results}\label{Sect3}
In this section, we state and prove the existence and uniqueness theorem. First, we prove
an existence theorem for a solution of the inverse problem. There are two main methods
in the literature to prove existence of the solution of inverse problems for the classical
diffusion equations: the monotonicity method \cite{TrPeSy, DuchTas,  Tar} and the quasi-solution method \cite{LiuTar, TrPeSy}.

\noindent In this article, we use the quasi-solution method. For this purpose, let \\ $(\beta,\alpha,\gamma)\in [\beta_0,\beta_1]\times[\alpha_0,\alpha_1]\times[\gamma_0,\gamma_1]$, where   $\beta_0, \alpha_0, \gamma_0>0$, $\alpha_1, \gamma_1<2$ and $ \beta_1<1$. For notation sake, we denote the unique solution to the direct problem corresponding to the parameter $(\beta,\alpha,\gamma)$ 
as $u(\beta,\alpha,\gamma)(t,x).$

For a given target function $\varphi\in L^2(0,T],$ we define the following minimization problem 
\begin{equation}\label{MinPrb}
\min\limits_{(\beta,\alpha,\gamma)\in [\beta_0,\beta_1]\times[\alpha_0,\alpha_1]\times[\gamma_0,\gamma_1]} {\Big\| u(\beta,\alpha,\gamma)(t,0)-\varphi(t)\Big\|}_{L^2(0,T]}.
\end{equation}

Next, define the input-output mapping

\begin{equation}\label{InOutMap}
\begin{split}
F(\beta,\alpha,\gamma)(t): [\beta_0,\beta_1]\times &[\alpha_0,\alpha_1]\times[\gamma_0,\gamma_1]\rightarrow L^2(0,T]\\
& (\beta,\alpha,\gamma)\mapsto u(\beta,\alpha,\gamma)(t,0).
\end{split}
\end{equation}
This mapping is well defined. To see this, use \eqref{WkSol} and Lemma \ref{Mtg2} to get 
\begin{align*}
\int_0^T |u(\beta,\alpha,\gamma)(t,0)|^2dt \leq  \int_0^T \sum_{\substack{n\geq 1\\ n \ \text{is odd}}}  \Big|E_\beta(-\mu_nt^\beta)\langle f, \psi_n\rangle \Big|^2 \leq C_0 \int_0^T {\|f\|}_{L^2(D)}^2 dt<\infty
\end{align*}
since $f\in L^2(D)$ and $0<T<\infty$.\\
We now proceed to prove a very important theorem about the input-output mapping \eqref{InOutMap}
\begin{theorem}\label{InOutThm}
The input-output mapping defined in \eqref{InOutMap} is Lipschitz continuous.
\end{theorem}
\begin{proof}
For each fixed $x$, we regard $u(t,x)$ as a mapping from $t\in[0,T]$ to $L^2(D).$ So we write $u(t):=u(t,\cdot).$ Pick    $(\hat{\beta},\hat{\alpha},\hat{\gamma})$ and $(\beta,\alpha,\gamma)$ two distinct points from the hyperrectangle $ [\beta_0,\beta_1]\times[\alpha_0,\alpha_1]\times[\gamma_0,\gamma_1]$. Without loss of generality, assume  $\hat{\beta}>\beta$ (the other cases follow similarly). Let $u= u(\beta,\alpha,\gamma)$, $v= u(\hat{\beta},\hat{\alpha},\hat{\gamma})$ and $y=u-v$.\\

\noindent It is not hard to see that $y$ solves the following initial value problem:
\begin{equation}\label{EqYsolv}
\begin{cases}
&\partial_t^\beta y= \Big[-(-\Delta)^{\alpha/2}-(-\Delta)^{\gamma/2}\Big]y + \underbrace{\Big[(-\Delta)^{\hat{\alpha}/2}+(-\Delta)^{\hat{\gamma}/2}\Big]v+\Big[-(-\Delta)^{\alpha/2}-(-\Delta)^{\gamma/2}\Big]v}_{:= I_1} \underbrace{-\partial_t^\beta v+\partial_t^{\hat{\beta}} v}_{:=I_2}\\
&y(0)=0.
\end{cases}
\end{equation}
We now estimate quantities $I_1$ and $I_2$. Using \eqref{FrLpDef} and the Mean Value Theorem, we get

\begin{equation}\label{I1Est}
\begin{split}
I_1= &\sum_{n=1}^\infty c_n\Big[ \Big(\bar{\mu}_n^{\hat{\alpha}/2}+\bar{\mu}_n^{\hat{\gamma}/2}\Big)-\Big(\bar{\mu}_n^{\alpha/2}+\bar{\mu}_n^{\gamma/2}\Big)\Big]\psi_n(x)\\
\leq & C_1\sum_{n=1}^\infty c_n \Big[\big|\hat{\alpha}-\alpha\big|\bar{\mu}_n^{\xi_1/2}+ \big|\hat{\gamma}-\gamma\big|\bar{\mu}_n^{\xi_2/2}\Big]\psi_n(x)\\
= & C_1 \Big[\big|\hat{\alpha}-\alpha\big|\sum_{n=1}^\infty c_n\bar{\mu}_n^{\xi_1/2}\psi_n(x) +\big|\hat{\gamma}-\gamma\big|\sum_{n=1}^\infty c_n\bar{\mu}_n^{\xi_2/2} \psi_n(x)\Big]\\
= & C_1\Big[\big|\hat{\alpha}-\alpha\big|\big(-\Delta\big)^{\xi_1/2}v+\big|\hat{\gamma}-\gamma\big|\big(-\Delta\big)^{\xi_2/2}v\Big],\\
\end{split}
\end{equation}
where $\xi_1$ is a number between $\alpha$ and $\hat{\alpha}$   and $\xi_2$ is a number between  $\gamma$ and $ \hat{\gamma}$. 
So by estimate \eqref{IndLplInq}, we have

\begin{equation}\label{EstNrmI1}
\begin{split}
{\|I_1(t)\|}_{L^2(D)}\leq &   C_1\Bigg[\big|\hat{\alpha}-\alpha\big|{\Big\|\big(-\Delta\big)^{\xi_1/2}v\Big\|}_{L^2(D)}+\big|\hat{\gamma}-\gamma\big|{\Big\|\big(-\Delta\big)^{\xi_2/2}v\Big\|}_{L^2(D)}\Bigg]\\
\leq & C_2 \Bigg[\big|\hat{\alpha}-\alpha\big| t^{-{\xi_1}}+ \big|\hat{\gamma}-\gamma\big| t^{-{\xi_2}}\Bigg]{\big\|f\big\|}_{L^2(D)}.
\end{split}
\end{equation}
Next, to estimate $I_2$, we write

$$
I_2= \underbrace{\Bigg[ 1-\frac{\Gamma(1-\hat{\beta})}{\Gamma(1-\beta)}\Bigg]\frac{1}{\Gamma(1-\hat{\beta})}\int_0^t (t-s)^{-\hat{\beta}} v'(s)ds}_{:=I_{21}} +\underbrace{ \frac{1}{\Gamma(1-\beta)}\int_0^t \Big[(t-s)^{-\hat{\beta}}-(t-s)^{-\beta} \Big]v'(s)ds}_{:=I_{22}}.
$$
Using the Lipschitz continuity of the Euler's gamma function, the fact that $x\mapsto\frac{1}{\Gamma(1-x)}$ is bounded on $(0,1)$ and \eqref{LplInq} and noting that $f\in L^2(D)$, the following estimates holds for $I_{21}:$

\begin{equation}\label{I21Est}
\begin{split}
{\big\|I_{21}\big\|}_{L^2(D)}\leq &C_3\Big|\Gamma(1-\hat{\beta})-\Gamma(1-\beta) \Big|{\Big\|\partial_t^{\hat{\beta}} v\Big\|}^2_{L^2(D)}\\
\leq & C_4 \Big|\hat{\beta}-\beta \Big|t^{-\hat{\beta}}, \ \ \ 0<t\leq T.
\end{split}
\end{equation}
Next, by definition,

$$
{\|v'\|}_{L^2(D)}= C_5 t^{\hat{\beta}-1}{\Bigg\|\sum_{n=1}^\infty \mu_n E_{\hat{\beta},\hat{\beta}}(-\mu_nt^{\hat{\beta}})\langle f, \psi_n\rangle \psi_n\Bigg\|}_{L^2(D)}
$$
Thus, since $\Big(\psi_n\Big)_{n\geq 1}$ is an orthonormal basis of $L^2(D)$, we get
\begin{equation*}
{\Bigg\|\sum_{n=1}^\infty \mu_n E_{\hat{\beta},\hat{\beta}}(-\mu_nt^{\hat{\beta}})\langle f, \psi_n\rangle \psi_n\Bigg\|}_{L^2(D)}\leq C_6 \Bigg(\sum_{n=1}^\infty\mu_n^2\langle f, \psi_n\rangle^2\Bigg)^{\frac{1}{2}}\leq C_7 {\Big\|\mathtt{L}^{\hat{\alpha}, \hat{\gamma}}_D f\Big\|}_{L^2(D)}\leq C_8{\|f\|}_{\dot{H}^{\hat{\alpha},\hat{\gamma}}}, \ \ \ 0<t\leq T.
\end{equation*}
Since $f\in \dot{H}^{\hat{\alpha},\hat{\gamma}}$, this implies that 

\begin{equation}\label{NrmVP}
{\|v'\|}_{L^2(D)}= C_9 t^{\hat{\beta}-1}, \ \ \ 0<t\leq T.
\end{equation}
We then proceed to get an estimate on $I_{22}:$

\begin{equation}\label{I22Est}
{\big\|I_{22}(t)\big\|}_{L^2(D)}\leq C_{10} \int_0^t \Big|(t-s)^{-\hat{\beta}}-(t-s)^{-\beta} \Big|{\|v'(s)\|}_{L^2(D)}ds\leq C_{11} \int_0^t \Big|(t-s)^{-\hat{\beta}}-(t-s)^{-\beta} \Big|s^{\hat{\beta}-1} ds \leq C_{12} \big|\hat{\beta}-\beta\big|,
\end{equation}
where the last inequality follows from \cite[Pages 16-17]{LiZhg}. Finally,

\begin{equation}\label{I2Est}
{\big\|I_2(t)\big\|}_{L^2(D)}\leq C_{13} \big|\hat{\beta}-\beta\big|\Big(1+t^{\hat{\beta}-1}\Big), \ \ \ 0<t\leq T.
\end{equation}
Next, for any function $z\in \dot{H}^{\alpha,\gamma}(D)$,  using Parseval identity and Lemma \ref{Mtg2} as well as \eqref{Vg}, we have


\begin{equation}\label{FrLpVz}
\begin{split}
{\Big\|V(t)z\Big\|}_{L^2(D)}= & {\Bigg\|\sum_{n=1}^\infty  t^{\beta-1}E_{\beta,\beta}(-\mu_nt^{\beta})\langle z, \psi_n\rangle\psi_n\Bigg\|}_{L^2(D)}\\
\leq  & C_{14}t^{\beta-1} \Bigg( \sum_{n=1}^\infty \langle z, \psi_n\rangle^2  \Bigg)^{\frac{1}{2}}\\
= & C_{14} t^{\beta-1}{\|z\|}_{L^2(D)}.
\end{split}
\end{equation}
We now solve \eqref{Eq1ODE} with $g\equiv 0$ and $h=I_1+I_2$ to get

\begin{equation*}
\begin{split}
{\Big\|y(t)\Big\|}_{L^2(D)}& ={\Bigg\|\int_0^t V(t-s)\Big(I_1(s)+ I_2(s)\Big)ds\Bigg\|}_{L^2(D)}\\
& \leq C_{15}\int_0^t (t-s)^{\beta-1}\Big({\big\|I_1(s)\big\|}_{L^2(D)} + {\big\|I_2(s)\big\|}_{L^2(D)}\Big)ds\\
&\leq C_{16}\int_0^t (t-s)^{\beta-1} \Bigg[|\hat{\alpha}-\alpha| s^{-\xi_1} + |\hat{\gamma}-\gamma| s^{-\xi_2}
+\big|\hat{\beta}-\beta\big|\Big(1+s^{\hat{\beta}-1}\Big) \Bigg]ds\\
=& C_{16}\Bigg[t^{\beta-\xi_1}\textbf{B}(\beta, 1-\xi_1)|\hat{\alpha}-\alpha|+t^{\beta-\xi_2}\textbf{B}(\beta, 1-\xi_2)|\hat{\gamma}-\gamma|+\Big(\beta^{-1}t^{\beta}+t^{\beta+\hat{\beta}-1}\textbf{B}(\beta, \hat{\beta})\Big)\big|\hat{\beta}-\beta\big|\Bigg],
\end{split}
\end{equation*}
where $\textbf{B}(a,b)=\int_0^1 w^{a-1}(1-w)^{b-1}dw$ is the Euler's Beta function for $a,b>0.$ Therefore there exists a positive constant C for which
$${\Big\|y(t)\Big\|}_{L^2(D)}\leq C\Big(\big|\hat{\beta}-\beta\big|+|\hat{\alpha}-\alpha|+\hat{\gamma}-\gamma|\Big) \ \ \text{for all} \ 0<t\leq T$$ and this concludes the proof.
\end{proof}
\noindent For practical use in the sequel, we define the following function

\begin{equation}\label{FuncI}
I(a):= {\big\|u(a)(t,0)-\varphi(t)\big\|}_{L^2(0,T]}^2,
\end{equation}
with $a=(\beta, \alpha, \gamma)\in [\beta_0,\beta_1]\times[\alpha_0,\alpha_1]\times[\gamma_0,\gamma_1].$

\noindent A usual argument on the compactness of the hyperrectangle $[\beta_0,\beta_1]\times[\alpha_0,\alpha_1]\times[\gamma_0,\gamma_1]\subset\mathbb{R}^3$ yields the following existence theorem.
\begin{theorem}
There exists $a^*= (\beta^*, \alpha^*, \gamma^*)\in [\beta_0,\beta_1]\times[\alpha_0,\alpha_1]\times[\gamma_0,\gamma_1]$ such that

$$
I(a^*)\leq I(a) \ \ \ \text{for all} \ \ a\in [\beta_0,\beta_1]\times[\alpha_0,\alpha_1]\times[\gamma_0,\gamma_1].
$$
\end{theorem}
 \noindent The next theorem provides the uniqueness of the solution to our inverse problem.
\begin{theorem}
Let $u$ be the weak solution of \eqref{eq1} and let $v$  be the weak solution of the following problem
\begin{equation}\label{UniqP}
\begin{cases}
&\partial_t^{\hat{\beta}} v(t,x)= -(-\Delta)^{\hat{\alpha}/2}v(t,x)- (-\Delta)^{\hat{\gamma}/2}v(t,x),  \ t>0, \ \ x\in D, \ \\
& v(t, -1) =  v(t,1)=0,   ~~0 < t \leq  T,\\
& v(0, x) = f(x), \ x\in D,
\end{cases}
\end{equation}
where $0<\hat{\beta}<1\ \ \text{and} \ \  0<\hat{\alpha}\leq\hat{\gamma}<2$.
$$ \text{If}\  u(t,0)=v(t,0), \ 0<t\leq T \ \text{ and} \   \eqref{RestF} \ \text{holds, then} \  \beta=\hat{\beta}, \ \  \alpha= \hat{\alpha} \ \ \text{and} \ \ \gamma=\hat{\gamma}.$$
\end{theorem}
\begin{proof}
	The proof follows a similar argument as in \cite{GuNaSoy}.
Using the explicit formula \eqref{WkSol}, the weak solutions $u$ and $v$ can be written as
\begin{equation}\label{PwExpU}
u(t,x)= \sum_{n=1}^{\infty} E_\beta(-\mu_nt^\beta)\langle f, \psi_n\rangle \psi_n(x)
\end{equation}
and
\begin{equation}\label{PwExpV}
v(t,x)= \sum_{n=1}^{\infty} E_{\hat{\beta}}(-\lambda_nt^{\hat{\beta}})\langle f, \psi_n\rangle \psi_n(x),
\end{equation}
where  $\mu_k= \bar{\mu}_k^{\alpha/2}+\bar{\mu}_k^{\gamma/2}$ and $\lambda_k=  \bar{\mu}_k^{\hat{\alpha}/2}+\bar{\mu}_k^{\hat{\gamma}/2}$ and the eigenpair $(\bar{\mu}_k, \psi_k)$ solves \eqref{Helmh}. 

\noindent Therefore, if we assume $u(t,0)=v(t,0) \ \ \text{for}\  \ 0<t\leq T,$ we get the following equation

\begin{equation}\label{EqWSol}
\sum_{\substack{n\geq 1\\ n \ \text{is odd}}} E_\beta(-\mu_nt^\beta)\langle f,  \psi_n\rangle= \sum_{\substack{n\geq 1\\ n \ \text{is odd}}} E_{\hat{\beta}}(-\lambda_nt^{\hat{\beta}})\langle f,  \psi_n\rangle, \ \ \ 0<t<T.
\end{equation}
Because both series in \eqref{EqWSol} are analytic in the domain $\Re t>0,$ it follows that

\begin{equation}\label{EqWSol1}
\sum_{\substack{n\geq 1\\ n \ \text{is odd}}} E_\beta(-\mu_nt^\beta)\langle f,  \psi_n\rangle= \sum_{\substack{n\geq 1\\ n \ \text{is odd}}} E_{\hat{\beta}}(-\lambda_nt^{\hat{\beta}})\langle f,  \psi_n\rangle, \ \ \ t>0.
\end{equation}
Next, we use the  asymptotic property of the Mittag-Leffler function  \eqref{Mitg1} to obtain, by adding and subtracting the term $\frac{1}{\Gamma(1-\beta)}\frac{1}{\mu_n t^\beta}$ in the left side term in \eqref{EqWSol},  the following asymptotic equation

\begin{equation}\label{AsyEqWSolU}
\begin{split}
\sum_{\substack{n\geq 1\\ n \ \text{is odd}}} E_\beta(-\mu_nt^{\beta})\langle f, \psi_n\rangle
=\sum_{\substack{n\geq 1\\ n \ \text{is odd}}}\langle f, \psi_n\rangle\frac{1}{\Gamma(1-\beta)}\frac{1}{\mu_n t^\beta} +O(|t|^{-2\beta}).
\end{split}
\end{equation}
Similarly,
\begin{equation}\label{AsyEqWSolV}
\begin{split}
\sum_{\substack{n\geq 1\\ n \ \text{is odd}}} E_{\hat{\beta}}(-\lambda_nt^{\hat{\beta}})\langle f, \psi_n\rangle
= \sum_{\substack{n\geq 1\\ n \ \text{is odd}}}\langle f, \psi_n\rangle\frac{1}{\Gamma(1-\hat{\beta})}\frac{1}{\lambda_n t^{\hat{\beta}}} +O(|t|^{-2\hat{\beta}}).
\end{split}
\end{equation}
Now combining \eqref{EqWSol1}, \eqref{AsyEqWSolU} and \eqref{AsyEqWSolV}, we get, as $t\rightarrow\infty$

\begin{equation}\label{AsymEqWSol}
\sum_{\substack{n\geq 1\\ n \ \text{is odd}}}\langle f, \psi_n\rangle\frac{1}{\Gamma(1-\beta)}\frac{1}{\mu_n t^\beta} +O(|t|^{-2\beta})= \sum_{\substack{n\geq 1\\ n \ \text{is odd}}}\langle f, \psi_n\rangle\frac{1}{\Gamma(1-\hat{\beta})}\frac{1}{\lambda_n t^{\hat{\beta}}} +O(|t|^{-2\hat{\beta}}).
\end{equation}
Now assume, for example, that $\beta>\hat{\beta}$. Then multiply \eqref{AsymEqWSol} by $t^{\hat{\beta}}$ to get
\begin{equation}\label{FinEqWsol}
t^{\hat{\beta}-\beta}\sum_{\substack{n\geq 1\\ n \ \text{is odd}}}\langle f, \psi_n\rangle\frac{1}{\Gamma(1-\beta)}\frac{1}{\mu_n } +O(|t|^{\hat{\beta}-2\beta})- \sum_{\substack{n\geq 1\\ n \ \text{is odd}}}\langle f, \psi_n\rangle\frac{1}{\Gamma(1-\hat{\beta})}\frac{1}{\lambda_n } +O(|t|^{-\hat{\beta}})=0.
\end{equation}
Letting $t\rightarrow\infty$ in \eqref{FinEqWsol} yields

\begin{equation}\label{Contrd}
\sum_{\substack{n\geq 1\\ n \ \text{is odd}}}\langle f, \psi_n\rangle\frac{1}{\Gamma(1-\hat{\beta})}\frac{1}{\lambda_n }=0: \ \ \text{a contradiction to \eqref{RestF}!}
\end{equation}
Similarly, assuming $\hat{\beta}>\beta$ also leads to a contradiction. Thus $\beta=\hat{\beta}.$

We now prove the second part of the Theorem, i.e $\alpha= \hat{\alpha}$ and $\gamma=\hat{\gamma}.$ To this aim, we will show that $\mu_n=\lambda_n$ for all $n\geq 1, \ n\ \text{is odd}.$ Since $\beta=\hat{\beta}$, \eqref{EqWSol} becomes

\begin{equation}\label{EqWSolStm}
\sum_{\substack{n\geq 1\\ n \ \text{is odd}}} E_\beta(-\mu_nt^\beta)\langle f, \psi_n\rangle= \sum_{\substack{n\geq 1\\ n \ \text{is odd}}} E_\beta(-\lambda_nt^\beta)\langle f, \psi_n\rangle.
\end{equation}
Taking the Laplace transform of $E_\beta(-\mu_nt^\beta)$ yields

\begin{equation}\label{LapTrf}
\int_0^\infty e^{-zt}E_\beta(-\mu_n t^\beta) dt= \frac{z^{\beta-1}}{z^\beta+\mu_n}, \ \ \Re z>0.
\end{equation}
Furthermore, taking the Laplace transform of the Mittag-Leffler function term by term,  we get

\begin{equation}\label{MitgLefTrmbT}
\int_0^\infty e^{-zt}E_\beta(-\mu_n t^\beta) dt= \frac{z^{\beta-1}}{z^\beta+\mu_n}, \ \ \Re z>\mu_n^{1/\beta}.
\end{equation}
Since $\sup\limits_{t\geq 0}\big|E_\beta(-\mu_nt^\beta)\big|<\infty$ by Lemma \ref{Mtg3}, this implies that $\int_0^\infty e^{-zt}E_\beta(-\mu_n t^\beta) dt$ is analytic in the domain $\Re z>\mu_n^{1/\beta}.$ Then by analytic continuity, $\int_0^\infty e^{-zt}E_\beta(-\mu_n t^\beta) dt$ is analytic in the domain $\Re z>0.$\\
Using Lemma \ref{Mtg3} and Lebesgue's convergence Theorem, we get that $$e^{-t\Re z} t^{-\beta} \ \ \text{is integrable for } \ \ t\in(0,\infty) \ \ \text{with fixed} \ z \ \ \text{such that} \ \Re z>0$$ and
\begin{align*}
\Big|e^{-t\Re z} \sum_{\substack{n\geq 1\\ n \ \text{is odd}}} E_\beta(-\mu_nt^\beta)\langle f, \psi_n\rangle\Big|\leq& C_0 e^{-t\Re z}\Bigg(\sum_{\substack{n\geq 1\\ n \ \text{is odd}}}\langle f, \psi_n\rangle \frac{1}{\mu_nt^{\beta}} \Bigg)\\
\leq & C_0^{'} e^{-t\Re z} t^{-\beta}\sum_{\substack{n\geq 1\\ n \ \text{is odd}}} \langle f, \psi_n\rangle <\infty.
\end{align*}

\noindent Next, for fixed $z$ satisfying $\Re z>0$, we have

\begin{equation}\label{LpTrEMu}
\int_0^\infty e^{-t z} \sum_{\substack{n\geq 1\\ n \ \text{is odd}}} E_\beta(-\mu_nt^\beta)\langle f, \psi_n\rangle dt= \sum_{\substack{n\geq 1\\ n \ \text{is odd}}}\langle f, \psi_n\rangle \frac{z^{\beta-1}}{z^\beta+\mu_n}.
\end{equation}
Similarly,

\begin{equation}\label{LpTrELbd}
\int_0^\infty e^{-t z} \sum_{n=1}^{\infty} E_\beta(-\lambda_nt^\beta)\langle f, \psi_n\rangle dt= \sum_{n=1}^{\infty}\langle f, \psi_n\rangle \frac{z^{\beta-1}}{z^\beta+\lambda_n}.
\end{equation}
This means, by \eqref{EqWSolStm}, \eqref{LpTrEMu} and \eqref{LpTrELbd},

\begin{equation}\label{LstEq}
\sum_{\substack{n\geq 1\\ n \ \text{is odd}}} \frac{\langle f, \psi_n\rangle}{\rho+\mu_n}= \sum_{\substack{n\geq 1\\ n \ \text{is odd}}} \frac{\langle f, \psi_n\rangle}{\rho+\lambda_n}, \ \ \Re \rho >0.
\end{equation}

\noindent The last equality is equivalent to 

\begin{equation}\label{LlastEq}
\sum_{\substack{n\geq 1\\ n \ \text{is odd}}}\frac{\langle f, \psi_n\rangle(\mu_n-\lambda_n)}{(\rho+\mu_n)(\rho+\lambda_n)}=0, \ \ \Re \rho >0
\end{equation}
i.e $\mu_n=\lambda_n$ for all $n\geq 1, \ n$ is odd. Next, by the definition (value) of $\mu_n$ and $\lambda_n$, we have 
\begin{equation}\label{EqEig}
\big(n\pi/2\big)^{\alpha}+\big(n\pi/2\big)^{\gamma}= \mu_n=\lambda_n=\big(n\pi/2\big)^{\hat{\alpha}}+\big(n\pi/2\big)^{\hat{\gamma}}.
\end{equation}

\noindent Therefore, defining $C_\upsilon:=(\pi/2)^\upsilon, \ \upsilon>0$, we have the following bounds for the eigenvalues: for all $n\geq 1, \ n$  odd,
\begin{equation}\label{EigVAl}
C_\alpha \big(n^\alpha+n^\gamma\big)\leq \mu_n\leq C_\gamma \big(n^\alpha+n^\gamma\big)
\end{equation}
and 
\begin{equation}\label{EigVhAl}
C_{\hat{\alpha}} \big(n^{\hat{\alpha}}+n^{\hat{\gamma}}\big)\leq \mu_n\leq C_{\hat{\gamma}} \big(n^{\hat{\alpha}}+n^{\hat{\gamma}}\big).
\end{equation}
Assume for example that $\gamma<\hat{\gamma}$, then combining equations \eqref{EigVAl} and \eqref{EigVhAl} gives, for positive constants $\kappa_1$ and $\kappa_2$,
$$
\kappa_1 n^{\hat{\gamma}}\leq \mu_n\leq \kappa_2n^{\gamma}  \ \ \text{for all} \ n\geq 1, \ n \ \text{odd}: \ \text{a contradiction!}
$$
Therefore $\gamma=\hat{\gamma}$ since the reverse inequality also leads to a contradiction. Finally by equation \eqref{EqEig}, this also means that $\alpha=\hat{\alpha}$  and this concludes the proof.
\end{proof}

\noindent We now describe the algorithm used to find the solution of our inverse problem.
\section{The inversion algorithm}\label{Sect4}

The inversion algorithm is based on the minimization of the error functional $I (a)$, which
is defined by \eqref{FuncI}. We note that the continuity, hence the existence
of the minimum of the functional on a compact set has been established in the previous
section which is not enough to set up an efficient search algorithm for the minimum. Before
developing an algorithm to find the minimum, we observe a key fact about the functional
$I (a)$ which is differentiability. Now we prove that under certain conditions on $f$ , $I (a)$ is differentiable with respect to $a$ on a neighbourhood of the minimum. This will enable us to implement a gradient method for the minimization.

\begin{theorem}\label{DiffI}
The function $I(a)$ is differentiable on $(0,1)\times (0,2)^2 $ if  $\big|\langle f, \psi_n\rangle\big|<n^{-(1+\alpha+\gamma+\theta)}$ for some $\theta>1$ and $\varphi$ is bounded.
\end{theorem}
\begin{proof}
Without loss of generality, take $T\equiv1.$ Recall that

\begin{align*}
I(a)= & {\big\|u(\beta, \alpha,\gamma)(t,0)-\varphi(t)\big\|}_{L^2(0,1]}^2\\
= & \int_0^1 u(\beta, \alpha,\gamma)^2(t,0)dt-2\int_0^1 u(\beta, \alpha,\gamma)(t,0)\varphi(t)dt+ \int_0^1 \varphi(t)^2dt.
\end{align*}
We now show that the integrands in the first two integrals are differentiable with respect to $\beta, \alpha$ and $\gamma$ and that the partial derivatives are continuous for each $0<t\leq 1.$\\
\noindent Recall that 
\begin{align*}
u(\beta, \alpha, \gamma)(t,0)=\sum_{\substack{n\geq 1\\ n \ \text{is odd}}} E_\beta(-\mu_nt^\beta)\langle f,  \psi_n\rangle= \sum_{\substack{n\geq 1\\ n \ \text{is odd}}} E_\beta\Bigg(-\Big[\Big(\frac{n\pi}{2}\Big)^\alpha+ \Big(\frac{n\pi}{2}\Big)^\gamma\Big]t^\beta\Bigg)\langle f,  \psi_n\rangle.
\end{align*}

Then 

\begin{equation}\label{ParAp}
\partial_\alpha u(\beta, \alpha, \gamma)(t,0)= -t^{\beta}\sum_{\substack{n\geq 1\\ n \ \text{is odd}}}\Big(\frac{n\pi}{2}\Big)^\alpha\ln\Big(\frac{n\pi}{2}\Big) E_\beta^{'}\Bigg(-\Big[\Big(\frac{n\pi}{2}\Big)^\alpha+ \Big(\frac{n\pi}{2}\Big)^\gamma\Big]t^\beta\Bigg)\langle f,  \psi_n\rangle.
\end{equation}
By Lemma \ref{Mtg3}, the continuity of $t^{\beta}$ on $(0,1]$ and the assumption on $\big|\langle f, \psi_n\rangle\big|$, there exists a positive constant $C$ such that

$$
\Bigg|-t^{\beta}\sum_{\substack{n\geq 1\\ n \ \text{is odd}}}\Big(\frac{n\pi}{2}\Big)^\alpha\ln\Big(\frac{n\pi}{2}\Big) E_\beta^{'}\Bigg(-\Big[\Big(\frac{n\pi}{2}\Big)^\alpha+ \Big(\frac{n\pi}{2}\Big)^\gamma\Big]t^\beta\Bigg)\langle f,  \psi_n\rangle \Bigg|<C \sum_{\substack{n\geq 1\\ n \ \text{is odd}}} \ln\Big(\frac{n\pi}{2}\Big)n^{-(1+\gamma+\theta)}<\infty.
$$
Thus, $\partial_\alpha u(\beta, \alpha, \gamma)(t,0)$ exists for each $t$ and is bounded on $(0,1]$. Recall that  $\varphi$ is also bounded on $(0,1]$. The continuity of $\partial_\alpha u(\beta, \alpha, \gamma)(t,0)$ is straightforward from \eqref{ParAp}.  We conclude that $I(\cdot)$ is differentiable with respect to $\alpha.$  A similar argument  shows that the functional $I(\cdot)$ is also differentiable with respect to $\gamma.$\\
\noindent We now show the differentiability with respect to $\beta.$ Note that 

$$
\partial_\beta u(\beta, \alpha, \gamma)(t,0)= \sum_{\substack{n\geq 1\\ n \ \text{is odd}}}\partial_\beta E_\beta\Bigg(-\Big[\Big(\frac{n\pi}{2}\Big)^\alpha+ \Big(\frac{n\pi}{2}\Big)^\gamma\Big]t^\beta\Bigg)\langle f,  \psi_n\rangle,
$$
with 

\begin{align*}
\partial_\beta E_\beta\Bigg(-\Big[\Big(\frac{n\pi}{2}\Big)^\alpha+ \Big(\frac{n\pi}{2}\Big)^\gamma\Big]t^\beta\Bigg)=& \sum_{k=1}^\infty\left(\frac{k\ln t\Big(-\Big[\Big(\frac{n\pi}{2}\Big)^\alpha+ \Big(\frac{n\pi}{2}\Big)^\gamma\Big]t^\beta\Big) \Big(-\Big[\Big(\frac{n\pi}{2}\Big)^\alpha+ \Big(\frac{n\pi}{2}\Big)^\gamma\Big]t^\beta\Big)^{k-1}}{\Gamma(\beta k+1)}\right)\\
&\qquad\qquad\qquad\qquad- \sum_{k=0}^\infty \left( \frac{k\Gamma'(\beta k+1)\Big(-\Big[\Big(\frac{n\pi}{2}\Big)^\alpha+ \Big(\frac{n\pi}{2}\Big)^\gamma\Big]t^\beta\Big)^{k}}{\Gamma(\beta k+1)^2}\right)\\
= & \underbrace{\Big(-\Big[\Big(\frac{n\pi}{2}\Big)^\alpha+ \Big(\frac{n\pi}{2}\Big)^\gamma\Big]t^\beta\Big)\ln t  E_\beta^{'}\Bigg(-\Big[\Big(\frac{n\pi}{2}\Big)^\alpha+ \Big(\frac{n\pi}{2}\Big)^\gamma\Big]t^\beta\Bigg)}_{:=A(t)}\\
& \qquad\qquad\qquad\qquad-\underbrace{\sum_{k=0}^{\infty} \frac{k\psi_0(1+\beta k)\Big(-\Big[\Big(\frac{n\pi}{2}\Big)^\alpha+ \Big(\frac{n\pi}{2}\Big)^\gamma\Big]t^\beta\Big)^{k}}{\Gamma(\beta k+1)}}_{:=B(t)},
\end{align*}
where $\psi_0(\cdot)= \frac{\Gamma'(\cdot)}{\Gamma(\cdot)}$ is the digamma function. Using Lemma \ref{Mtg3}, there exists some $M>0$ such that 
$$
\big|A(t)\big|<M\Big(\Big[\Big(\frac{n\pi}{2}\Big)^\alpha+ \Big(\frac{n\pi}{2}\Big)^\gamma\Big]t^\beta\Big)\big|\ln t\big| \ \ \text{on}\ \ (0,1].
$$ 
Next, since $\psi_0(1+\beta k)\approx\ln(1+\beta k)\leq k-1$ for sufficiently large $k,$ it must be the case that $B(t)$ remains bounded between some multiple of $n^{\gamma} t^{\beta}E_\beta^{'}\Bigg(-\Big[\Big(\frac{n\pi}{2}\Big)^\alpha+ \Big(\frac{n\pi}{2}\Big)^\gamma\Big]t^\beta\Bigg)$ and some multiple of $n^{2\gamma} t^{2\beta}E_\beta^{''}\Bigg(-\Big[\Big(\frac{n\pi}{2}\Big)^\alpha+ \Big(\frac{n\pi}{2}\Big)^\gamma\Big]t^\beta\Bigg)$; i.e for some $N>0,$

$$
\big|B(t)\big|< N n^{2\gamma} \ \ \ \text{on} \ \ (0,1].
$$

Whence,

\begin{align*}
\Big|\partial_\beta u(\beta, \alpha, \gamma)(t,0)\Big|<& \sum_{\substack{n\geq 1\\ n \ \text{is odd}}}\frac{\big|A(t)\big|+\big|B(t)\big|}{n^{1+\alpha+\gamma+\theta}}\\
<& C \sum_{\substack{n\geq 1\\ n \ \text{is odd}}}\frac{n^\gamma t^\beta \big|\ln t\big|+ n^{2\gamma}}{n^{1+\alpha+\gamma+\theta}}\\
<& C_1 t^{\beta}\big|\ln t\big|+C_2.
\end{align*}
Hence, $\partial_\beta u(\beta, \alpha, \gamma)(t,0)\in L^1(0,1].$ This fact combined with the boundedness of $u(\beta, \alpha, \gamma)(t,0)$ and $\varphi$ imply that the derivative with respect to $\beta$ of $u(\beta, \alpha, \gamma)(t,0)$ exists. Finally, the continuity of the partial derivative with respect to $\beta$ follows from the continuity $A(t)$ and $B(t)$ with respect to $\beta$ for each $t.$  This concludes the proof.
\end{proof}
We address the ill-posedness of our parameter estimation problem by adding a Tikhonov regularization term to the cost functional $I(a)$. It can then be formulated as the constrained nonlinear least squares problem
\begin{equation}\label{NLS}
\min_{a \in \Gamma} \frac{1}{2}\|u(a)(t,0)-\varphi(t)\|^2_{L^2(0,T)} + \frac{\lambda}{2}\|a\|_2^2,
\end{equation}
where $\Gamma = [\beta_0,\beta_1]\times [\alpha_0,\alpha_1]\times [\gamma_0,\gamma_1]$, $\|\cdot\|_2$ denotes the Euclidean norm, and $\lambda>0$ is a suitably chosen regularization parameter. 
The regularization term improves the stability of the minimizer $a_\lambda$ in the presence of measurement noise at the cost of biasing the estimate. Heuristic methods are typically used to choose the parameter $\lambda$ that balances these two errors, the most well-known of which is the Morozov discrepancy principle. Specifically, let $a^*$ be the true parameter value and suppose the measurement error $\varepsilon>0$ is known, i.e. $I(a^*)\leq \varepsilon$. According to the Morozov principle, $\lambda$ should be such that $I(a_\lambda) \approx \varepsilon$, i.e. the regularized solution need only be accurate to within the noise level.

In our numerical computations, we approximate $I(a)$ by a quadrature rule with nodes $0<t_1<...<t_m=T$ and weights $w_1,...,w_m \geq 0$, resulting in
\begin{equation}\label{IAPPROX}
I(a) \approx \frac{1}{2}\sum_{i=1}^m w_i \left[u(a)(t_i,0)-\varphi(t_i)\right]^2
\end{equation}
Defining the weighted residual vector $r(a)=[r_1(a),...,r_m(a)]^T\in \mathbb{R}^m$ componentwise by $r_i(a) = \sqrt{w_i}\left[u(a)(t_i,0)-\varphi(t_i)\right]$, we can approximate Problem \eqref{NLS} by the semi-discretized box-constrained nonlinear least squares problem 
\begin{equation}\label{NLSD}
\min_{a \in \Gamma} F(a) = \frac{1}{2}\|r(a)\|_2^2 + \frac{\lambda}{2}\|a\|_2^2.
\end{equation}  
We solve this problem by a trust-region method with trust region defined in terms of the $\ell_\infty$-norm: At every iteration step $k$, a quadratic model function $m_k(p)$ is constructed to approximate $F(a_k+p)$ within a region of the current iterate $a_k$. This model is then minimized, subject to the intersection of the box constraint $a_k+p\in \Gamma$ and the trust region constraint $\|p\|_\infty \leq R_k$, where $R_k$ is the trust-region radius at the kth step, i.e.
\begin{equation}\label{TRSP}
\min_{p} m_k(p), \qquad \text{subject to}\qquad \|p\|_\infty \leq R_k, \ a_k+p \in \Gamma,  
\end{equation} 
also known as the trust region subproblem. Once the minimizer $p_k$ (or at least an approximation thereof) is found, the decrease predicted by the model is compared with the actual decrease of the functional to determine (i) whether to accept the update $a_{k+1} = a_k + p_k$, and (ii) whether to adjust the trust region radius at the next step. The specifics of the algorithm are given below in Algorithm \ref{TRALG}.

\begin{algorithm}[ht!]
\SetAlgoLined
\KwIn{$a_0$, $R_{\max}$, $R_0\in(0,R_{\max}]$, $\eta\in [0,\frac{1}{4})$}
\For{$k=0,1,...$}{
Compute $r(a_k)$ and $J(a_k)$\;
Determine minimizer $p_k$ of Problem \eqref{TRSP}\;
Evaluate $\displaystyle \rho_k = \frac{F(x_k)-F(x_k+p_k)}{m_k(0)-m_k(p_k)}$\;
\tcc{Update trust region radius}
\uIf{$\rho_k<\frac{1}{4}$}{
$R_{k+1} = \frac{1}{4}R_k$\;}
\uElseIf{$\rho_k>\frac{3}{4}$ and $\|p_k\|_{\infty}=R_k$}{
$R_{k+1}=\min(2R_k, R_{\max})$\;}
\uElse{$R_{k+1}=R_k$\;}
\tcc{Update $a_k$}
\uIf{$\rho_k>\eta$}{
$a_{k+1} = a_k + p_k$}
\uElse{$a_{k+1} = a_k$}
}
\caption{Constrained trust-region least squares algorithm}\label{TRALG}
\end{algorithm}

The quadratic model function is commonly based on the second order Taylor expansion of $F$ about the current iterate. Let $J(a)=[\nabla r_1(a)^T,...,\nabla r_m(a)^T]^T$ be the Jacobian matrix. Then the gradient and Hessian of $F$ are given by 
\begin{align*}
\nabla F(a) &= \sum_{i=1}^m r_i(a) \nabla r_i(a) + \lambda a = J(a)^Tr(a) + \lambda a, \ \ \text{and}\\
\nabla^2 F(a) &= \sum_{i=1}^m \nabla r_i(a)\nabla r_i(a)^T + \sum_{i=1}^m r_i(a) \nabla^2r_i(a) + \lambda I \\
&= J(a)^T J(a) + \lambda I + \sum_{i=1}^m r_i(a) \nabla^2r_i(a).
\end{align*}
To avoid computing the second derivative of the residuals, we make use of the well-known Levenberg-Marquardt approximation
\begin{equation}\label{LM}
\nabla^2 F(a) \approx J(a)^TJ(a) + \lambda I.
\end{equation}

The approximation \eqref{LM} is accurate in general when the residuals are small and/or only slightly nonlinear in $a$. Note, however that a good approximation, while ensuring faster descent, is not necessary for the convergence of this method. Indeed, when the current model does not yield a sufficient decrease in $F$, $\rho_k$ is small and consequently the trust region radius is decreased (see Algorithm \ref{TRALG}), resulting in a smaller region within which the second order terms are less significant. In our numerical experiments we nevertheless found there to be good agreement between the model function and $F$.

We now briefly discuss the solution of the trust region subproblem \eqref{TRSP}. Since both constraints $\|p\|_\infty<R_k$ and $x_k+p \in \Gamma$ amount to componentwise bounds on $p$ (also known as box constraints), their intersection has the same form. We first compute the unconstrained minimizer $\tilde p_k$ for the model function $m_k(p)$ on $\mathbb{R}^3$. By virtue of the regularization term, the approximation \eqref{LM} is always positive definite, ensuring that $\tilde p_k$ exists and is unique. If $\tilde p_k$ satisfies the constraints, then $p_k=\tilde p_k$. Otherwise, we compute the constrained minimizer $p_k$ by projecting $\tilde p_k$ onto the box, thereby fixing at least one component, and minimizing $m_k(p)$ over the lower dimensional box bounding the remaining components. This proceedure is computationally inexpensive, since it does not require us to re-solve the double fractional PDE, and is guaranteed to terminate after at most 3 steps.
 
We terminate the algorithm either (i) when the maximum number of iterations are reached, or (ii) when the norm of the gradient of the Lagrange functional associated with Problem \eqref{NLSD} is within a predetermined tolerance level.   

\section{Numerical examples with noise-free and noisy data}\label{Sect5}
In this section we conduct numerical experiments to explore properties of the minimizer, to investigate the performance of the proposed optimization algorithm, and to determine the effect of measurement noise and the initial guess on the parameter estimates. In each example, we compute the weak solution \eqref{WkSol} $u(a^*)(t,x)$ for a known parameter value $a^*$ and construct the observation data $\varphi(t_i)$, $i=1,...,m$, by adding a uniformly distributed random noise vector, i.e.
\begin{equation}\label{NOISYOBS}
\varphi(t_i) = u(a^*)(t_i,0) + \frac{\delta}{\|u(a^*)\|_{L^2(0,T)}} \xi_i,
\end{equation}
where the perturbations $\xi_i \sim U(-1,1)$ are independent and identically distributed and $\delta \geq 0$. We compute the numerical approximation \eqref{IAPPROX} of $I(a)$ by the trapezoidal rule.

The computation of the weak solution requires estimating the Mittag-Leffler function \eqref{mtglfr} and the components of the initial condition onto the spectral basis, as well as determining an adequate truncation level for approximating the spectral expansion. In Example \ref{EX1}, we eliminate the error caused by the latter two approximations by choosing an initial condition that is a linear combination of eigenfunctions. We examine the effect of the spectral truncation error in Example \ref{EX2}. We use a numerically stable approximation of the Mittag-Leffler function, examined in \cite{Gorenflo2002}. 

The Jacobian function $J$ needed for the quadratic model $m_k(p)$ is approximated by difference quotients. In particular,  
\[
\frac{\partial r_i(a)}{\partial \beta} \approx \frac{u(\beta+d\beta,\alpha,\gamma)(t_i,0)-u(\beta,\alpha,\gamma)(t_i,0)}{d\beta}, \qquad 0<d\beta \ll 1,
\] 
with similar approximations for $\frac{\partial r_i(a)}{\partial \alpha}$ and $\frac{\partial r_i(a)}{\partial \gamma}$. In our computations we choose perturbations $d\beta=d\alpha=d\gamma=10^{-7}$, to ensure sufficient accuracy while avoiding roundoff error. 

\begin{example}\label{EX1}
In this example, we consider Problem \eqref{eq1} with initial condition
\[
f(x) = \cos\left(\frac{\pi}{2}x\right) + \frac{1}{2} \cos\left(\frac{5\pi}{2}x\right).
\] 
The weak solution \eqref{WkSol} can therefore be written explicitly as 
\[
u(t,x) = E_{\beta}\left(-\mu_1 t^\beta \right)\cos\left(\frac{\pi}{2}x\right) + \frac{1}{2}E_{\beta}\left(-\mu_5 t^\beta\right)\cos\left(\frac{5\pi}{2}x\right),
\]
where $\mu_k = \left(\frac{k\pi}{2}\right)^\alpha + \left(\frac{k\pi}{2}\right)^\gamma$ for $k=1,5$. We first test the performance of our algorithm on noiseless observations, i.e. we choose $\delta=0$ in Expression \eqref{NOISYOBS} and regularization parameter $\lambda=10^{-7}$. Figure \ref{EX1_CONV} shows convergence of the algorithm after 6 steps, both in terms of $F(a)$ and the norm of the Lagrangian gradient $\|\nabla L(a,\ell)\|$, while Figure \ref{EX1_REPR} shows the difference between the measurement $\varphi(t)$ and the model output at various iterations.

\begin{figure}[th!]
    \centering
    \begin{subfigure}[t]{0.5\textwidth}
        \centering
        \includegraphics[scale=1]{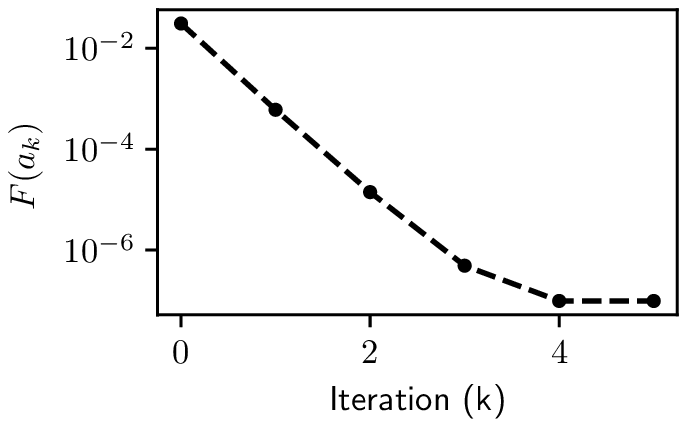}
        \caption{Semilog plot of cost functional.}
    \end{subfigure}%
    ~ 
    \begin{subfigure}[t]{0.5\textwidth}
        \centering
        \includegraphics[scale=1]{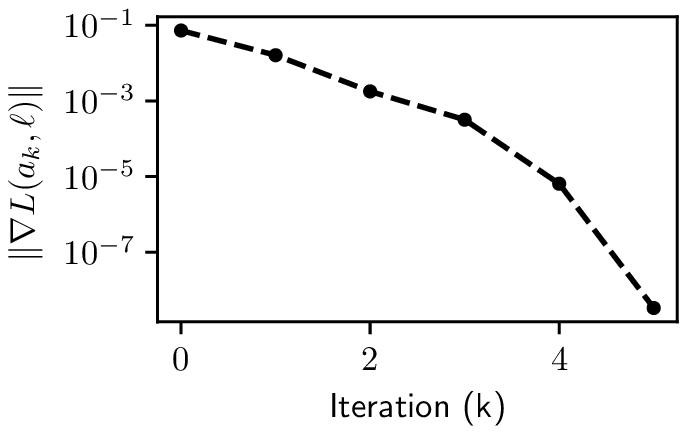}
        \caption{Semilog plot of the norm of the Lagrangian.}
    \end{subfigure}
    \caption{Convergence plots for Example \ref{EX1} in the noisefree case with initial guess $\alpha_0=0.1$, $\beta_0=0.05$, and $\gamma_0=1.7$. The exact parameter values are $\alpha^*=0.6$, $\beta^*=0.4$, and $\gamma^*=1.2$, while the estimated values are $\alpha_5=0.6$, $\beta_5=0.4005$, and $\gamma_5=1.1998$.}\label{EX1_CONV}
\end{figure}

\begin{figure}[ht!]
	\centering
	\includegraphics[scale=1]{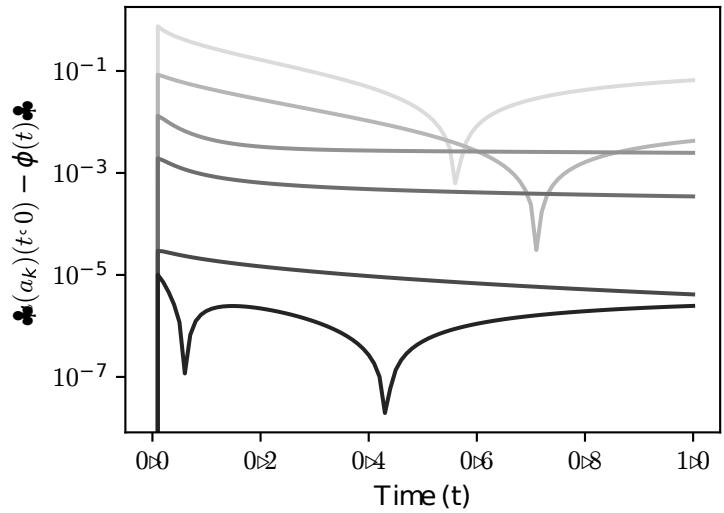}
	\caption{The discrepancy between measurement $\varphi(t)$ and solution $u(a_k)(t,0)$ at various iterations of the optimization algorithm. The shade of the curves indicate the iteration index, with lighter curves representing earlier iterations and darker curves later ones.}\label{EX1_REPR}
\end{figure}

To determine the influence of the initial guess on both the computed minimizer and on the optimization iteration itself, we chose 24 initial guesses $a_0$ randomly within the parameter region $\Gamma$. The convergence behavior for each guess is summarized in Figure \ref{EX1_STARTS}. In all cases, the algorithm converges to the same point within fewer than 9 iterations. 

\begin{figure}[ht!]
	\centering
	\includegraphics[scale=1]{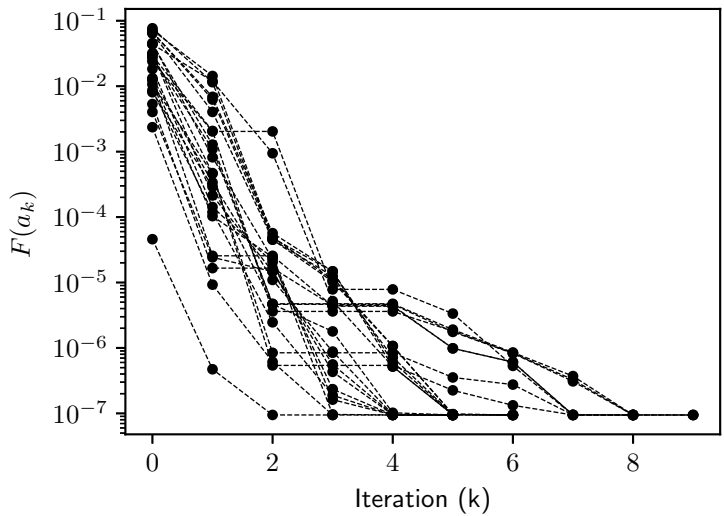}
	\caption{The convergence of the cost functional in Example \ref{EX1} for 24 randomly chosen initial guesses.}\label{EX1_STARTS}
\end{figure}

Next we add a 50\% noise, i.e. $\delta = 0.5$ in Equation \eqref{NOISYOBS} to the observation $\phi(t)$ and investigate the deviation $I(a_\lambda^*)$ as the regularization parameter is decreased. The results are shown in Figure \ref{EX1_NOISE}. 

\begin{figure}[ht!]
\centering
\begin{subfigure}[t]{0.48\textwidth}
\includegraphics[width=\textwidth]{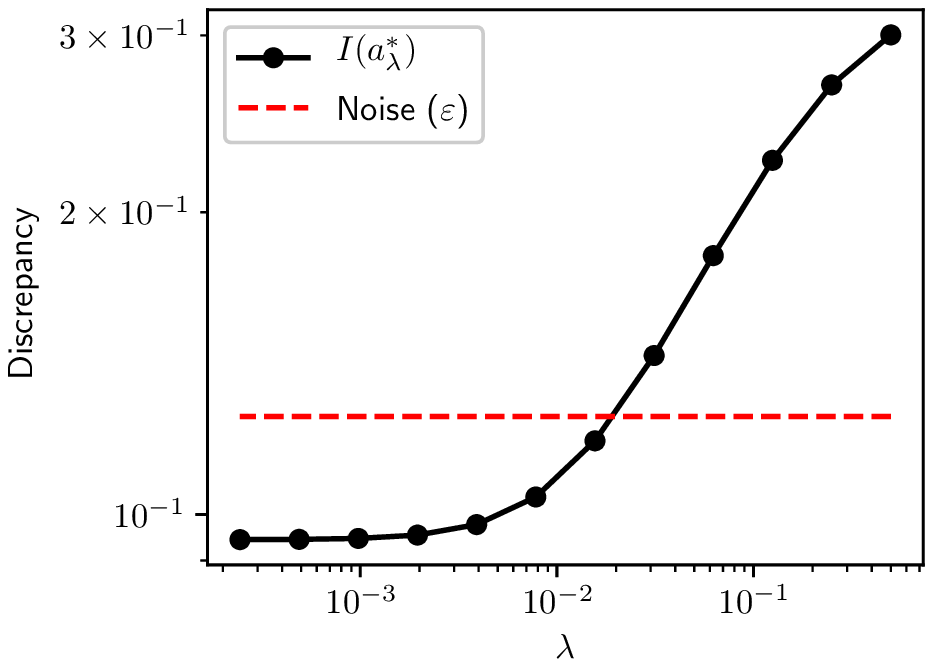}
\caption{Discrepancy $I(a_\lambda^*)$ for various values of $\lambda$.}\label{EX1_OPTREG}
\end{subfigure}~
\begin{subfigure}[t]{0.48\textwidth}
\includegraphics[width=\textwidth]{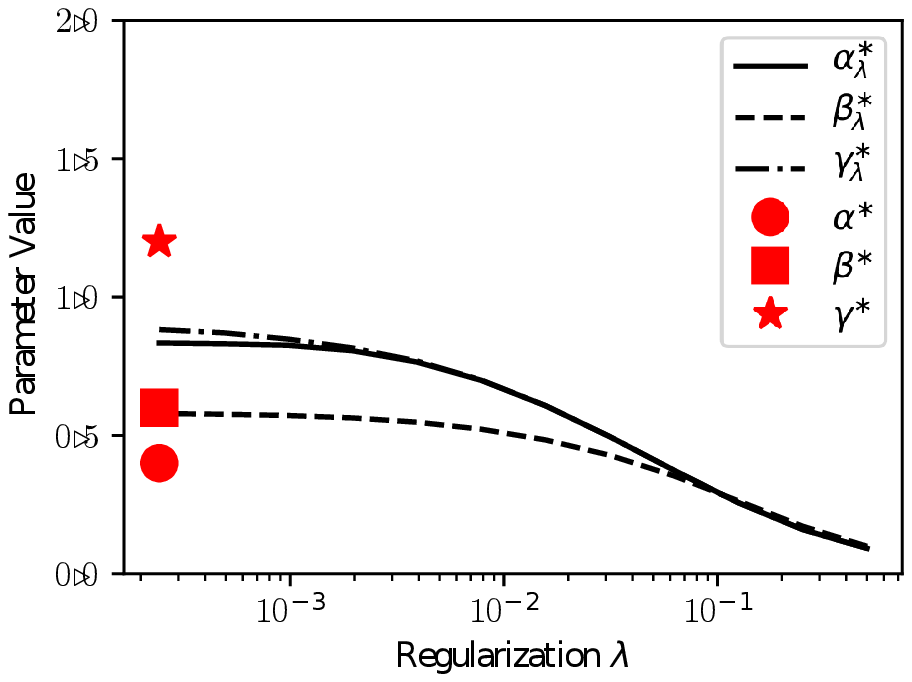}
\caption{Estimated parameter values as a function of the regularization $\lambda$.}\label{EX1_PARVALS}
\end{subfigure}
\caption{Parameter estimation with 50\% noise.}
\label{EX1_NOISE}
\end{figure}

Figure \ref{EX1_OPTREG} shows that the discrepancy between model output and target initially decreases as $\lambda$ is lowered, but flattens off eventually, at a level much larger than in the noiseless case. According to the Morozov discrepancy principle, $\lambda=\frac{1}{64}$ should be chosen as the regularization parameter. Figure \ref{EX1_PARVALS} shows how the optimal parameter values change with $\lambda$. It is interesting to note that, while $\beta$ can be readily identified, the powers of the double fractional Laplacian are not very accurate. The observation can 
can nevertheless be reconstructed well, as shown in Figure \ref{EX1_RECONSTR}.
\begin{figure}[ht!]
\centering
\includegraphics[scale=1]{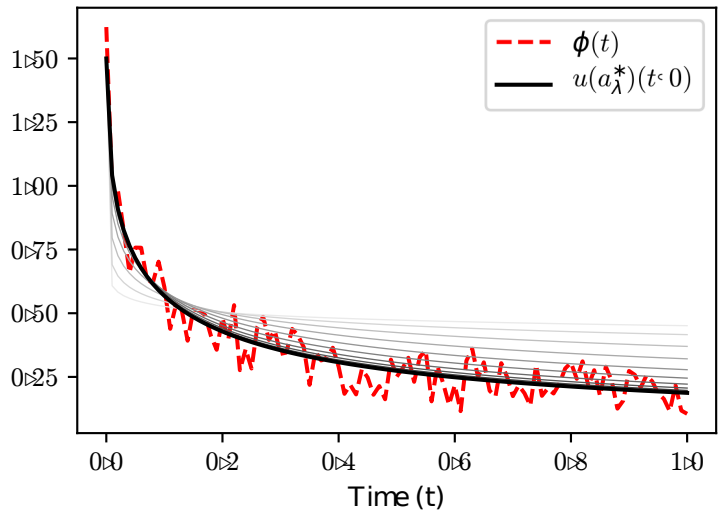}
\caption{Reconstructed data at noise level $\delta=\frac{1}{2}$ as $\lambda$ decreases from $1$ to $2^{-12}$. Lighter shades indicate larger values of $\lambda$.}\label{EX1_RECONSTR}
\end{figure}

\end{example}

\begin{example}\label{EX2}
To investigate the effect of truncation error, we consider Problem \eqref{eq1} with initial condition 
\[
f(x)=e^{-x^2}-e^{-1}.
\]
The addition of the constant term ensures that $f(x)$ satisfies the homogeneous Dirichlet boundary conditions. Unlike before, the spectral expansion is neither finite, nor can its components be computed exactly. We use an adaptive quadrature rule with an error tolerance of $10^{-8}$ to compute the components, so that the truncation level now constitutes the main source of error. In Figure \ref{EX2_TRUNC}, we show how the truncation level affects the accuracy of the noisefree estimates. It is evident from Figure \ref{EX2_RECONSTR} that, even at the true parameter $a^*$,  truncation reduces the accuracy in reconstructing the measurement $\varphi(t)$, with the largest error occurring at smaller values of $t$.
\begin{figure}[ht!]
\centering
\begin{subfigure}[t]{0.45\textwidth}
\includegraphics[width=\textwidth]{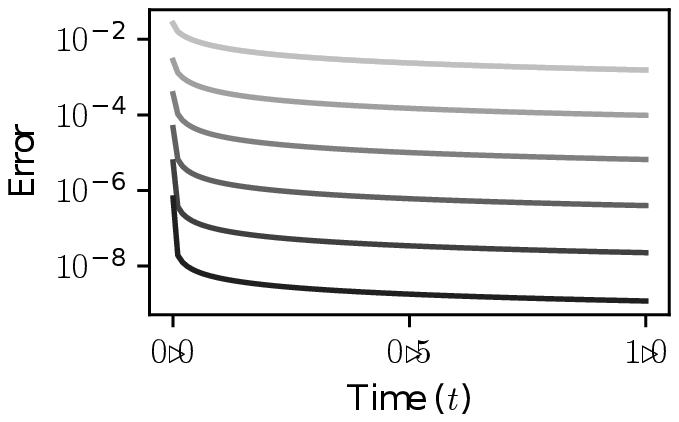}
\caption{Error in reconstructing $\varphi$ at exact parameter values. Darker curves correspond to more expansion terms.}\label{EX2_RECONSTR}
\end{subfigure}~~
\begin{subfigure}[t]{0.45\textwidth}
\includegraphics[width=\textwidth]{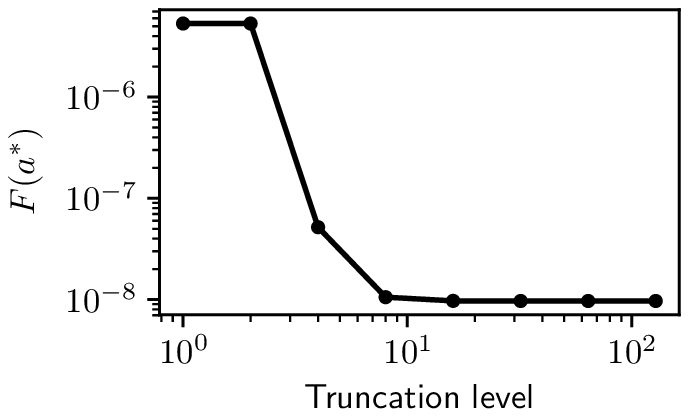}
\caption{Objective function value at minimizer for various truncation levels.}\label{EX2_OBJ}
\end{subfigure}
\caption{Accuracy of the optimal parameter $a^*$ for various truncation levels.}\label{EX2_TRUNC}
\end{figure}

Figure \ref{EX2_OBJ} shows that the objective function at the minimizer decreases to the level of $\lambda$ as the truncation level is increased. Figure \ref{EX2_PARVALS} describes the change in estimated parameter values as the trunctation level increases. While the time-fractional parameter $\beta$ is estimated accurately throughout, the fractional powers of $\alpha$ and $\gamma$ of the double Laplacian are only identified accurately at a sufficiently high truncation level. Interestingly, the minimizers $\alpha^*$ and $\gamma^*$ tend to lie close together at lower truncation levels, similar to the noisy case in Example \ref{EX1} (c.f. Figure \ref{EX1_PARVALS}). This suggests that, in the presence of noise or error, estimating the double fractional Laplacian by a single `average' fractional Laplacian gives a sufficiently good reconstruction of the data.   

\begin{figure}[ht!]
\centering
\includegraphics[scale=1]{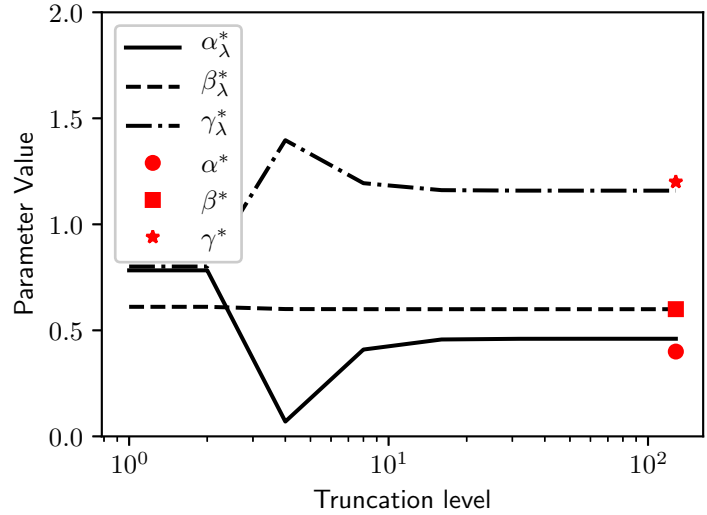}
\caption{Estimated parameter values at various truncation levels.}\label{EX2_PARVALS}
\end{figure}
\end{example}
\newpage
\section{Conclusion}
We have studied a nonlocal inverse problem for the space-time fractional diffusion
$$\partial_t^\beta u(t, x) = -(-\Delta)^{\alpha/2}u(t,x) - (-\Delta)^{\gamma/2}u(t,x) \ \ t\geq 0, \ -1<x<1. $$ After defining the input–output mapping for the inverse
problem, we have proved that the mapping is continuous. By using continuity of the mapping and compactness of the hyperrectangle $[\beta_0,\beta_1]\times[\alpha_0,\alpha_1]\times[\gamma_0,\gamma_1]$,  we have concluded that the minimization problem has a solution. The uniqueness of the solution has been proved for a specific class of the initial functions $f (x)$ using eigenfunction expansion of the solution of the direct problem. For the numerical solution of the inverse problem, a numerical method based on trust-region method and least squares approach are proposed. The numerical algorithm determines the unknowns $\beta, \alpha$ and $\gamma$ simultaneously.




\newpage
\Addresses

\end{document}